\documentclass[12pt]{amsart}
\usepackage{amsmath, amsthm, amscd, amsfonts}

\makeatletter \oddsidemargin.9375in \evensidemargin \oddsidemargin
\newtheorem{theorem}{Theorem}[section]
\newtheorem{lemma}[theorem]{Lemma}

\newtheorem{corollary}[theorem]{Corollary}
\theoremstyle{definition}
\newtheorem{definition}[theorem]{Definition}

\theoremstyle{remark}
\newtheorem{remark}[theorem]{Remark}
\numberwithin{equation}{section}

\newcommand{\R}{{\mathbb R}}

\begin{document}

\title[On the modified entropy equation]
{On the modified entropy equation}

\author[Eszter Gselmann ]{Eszter Gselmann }

\address{Institute of Mathematics, University of Debrecen, 4010 Debrecen,
P. O. Box 12, Hungary}

\email{gselmann@math.klte.hu}




\date{Received: \today,  Revised:}

\begin{abstract}
The object of this paper is to solve the so--called
modified entropy equation
\[
f\left(x, y, z\right)=f\left(x, y+z, \mathbf{0}\right)+
\mu\left(y+z\right)f\left(\mathbf{0}, \frac{y}{y+z}, \frac{z}{y+z}\right),
\]
on the positive cone of $\R^{k}$, where $\mu$ is a given multiplicative
function on this cone. After that the regular solutions of this equation are determined.
Furthermore we investigate its connection between the entropy equation and other
equations, as well.
\end{abstract} \maketitle

\section{Introduction and preliminaries}
Throughout this paper we will use the following notations
\[
\R^{k}_{+}=\left\{x\in\R^{k}\vert x\geq \mathbf{0}\right\} \quad \text{and} \quad
\R^{k}_{++}=\left\{x\in\R\vert x>\mathbf{0}\right\},
\]
where $k$ is an arbitrary but fixed positive integer and $\R$ denote the set of the real numbers.
Furthermore all operations on vectors are to be done
\textbf{componentwise}, i.e., $x\geq\mathbf{0}$ and $x>\mathbf{0}$ denotes that all coordinates of the vector
$x$ are nonnegative and positive, respectively. Here $\mathbf{0}$ stands for the $k$-vector
$\left(0, \ldots, 0\right)\in\R^{k}$ and we write $\mathbf{1}$ instead of $\left(1, \ldots, 1\right)\in\R^{k}$,
and $\mathbf{\frac{1}{2}}$ instead of $\left(\frac{1}{2}, \ldots, \frac{1}{2}\right)\in\R^{k}$. \\
\noindent
In this paper the general and the regular solutions of the
equation
\begin{equation}\label{gen.ent.eq}
f\left(x, y, z\right)=f\left(x, y+z, \mathbf{0}\right)+
\mu\left(y+z\right)f\left(\mathbf{0}, \frac{y}{y+z}, \frac{z}{y+z}\right),
\end{equation}
are determined on the positive cone of $\R^{k}$ with some
$\mu:\R^{k}_{++}\rightarrow\R$ multiplicative
function. Equation (\ref{gen.ent.eq})
is called \emph{the modified entropy equation}.
In \cite{AGMS08} the symmetric solutions of the equation
\begin{equation}\label{ent.eq}
f\left(x, y, z\right)=
f\left(x, y+z, 0\right)+\left(y+z\right)f\left(0, \frac{y}{y+z}, \frac{z}{y+z}\right)
\end{equation}
are determined, which is supposed to hold for all  $x, y, z\in\R_{++}$.
However, in \cite{KM74} and also in \cite{Acz77} a
similar equation, namely the entropy equation was discussed, more precisely, under some
assumptions the solution of the equation
\begin{equation}\label{entropy.eq}
f\left(x, y, z\right)=f\left(x+y, z, 0\right)+f\left(x, y, 0\right)
\end{equation}
on the set $\left\{\left(x, y, z\right)\in\R^{3}\vert x\geq 0, y\geq 0, z\geq 0, xy+yz+zx>0\right\}$
was determined. \\
In the third part of our paper we will show that equations
(\ref{ent.eq}) and (\ref{entropy.eq}) are special cases of equation (\ref{gen.ent.eq}). \\
\noindent
In what follows we shall
list some basic facts about the theory of functional equations. These
results can be found for instance in \cite{Kuc85} or in \cite{RB87}.

\begin{definition} \cite{RB87}
Let $I\subset\R^{k}_{+}$ and
\[
\mathcal{A}=\left\{(x, y)\in\R^{2k}_{+}\vert x, y, x+y\in I\right\}.
\]
A function
$a:I\rightarrow\R$ is called \textbf{additive on $\mathcal{A}$} if
\begin{equation}\label{add}
a\left(x+y\right)=a\left(x\right)+a\left(y\right)
\end{equation}
holds for all pairs $(x, y)\in \mathcal{A}$. \\
\noindent
Consider the set
\[
\mathcal{I}=\left\{(x, y)\in\R^{2k}_{+}\vert x, y, xy \in I\right\}.
\]
We say that
$\mu:I\rightarrow\R$ is \textbf{multiplicative on $\mathcal{I}$} if the functional equation
\begin{equation}\label{mult}
\mu\left(x y\right)=\mu\left(x\right)\mu\left(y\right)
\end{equation}
is fulfilled for all $(x, y)\in\mathcal{I}$. \\
\noindent
If
\[\mathcal{L}=\left\{(x, y)\in\R^{2k}_{++}\vert x, y, xy \in I\right\}
\]
then a function
$l:I\rightarrow\R$ is called \textbf{logarithmic on $\mathcal{L}$} if it satisfies the functional equation
\begin{equation}\label{log}
l\left(x y\right)=l\left(x\right)+l\left(y\right)
\end{equation}
for all $(x, y)\in \mathcal{L}$. \\
\noindent
In case $\mathcal{P}=\mathcal{A}\cap \mathcal{I}$, a
 function $\pi:I\rightarrow\R$ which is both additive and multiplicative on $\mathcal{P}$
is called \textbf{projection}.
\end{definition}
\noindent
Concerning extensions of multiplicative and logarithmic functions
we shall use the following two theorems.

\begin{theorem}\label{ext.mult} \cite{Kuc85}
Every multiplicative function $\mu:]0, 1[^{k}\rightarrow\R$
is uniquely extendable to a multiplicative function
$\tilde{\mu}:\R^{k}_{+}\rightarrow\R$.
\end{theorem}

\begin{theorem}\label{ext.log}\cite{Kuc85}
Every logarithmic function $l:]0, 1[^{k}\rightarrow\R$
is uni\-quely extendable to a logarithmic function
$\tilde{l}:\R^{k}_{++}\rightarrow\R$.
\end{theorem}
\noindent
Due to Theorems \ref{ext.mult}. and \ref{ext.log}. multiplicative and logarithmic functions
occurring in the subsequent section carry such global meaning. \\
\noindent
During the proof of our main result we will use the following lemmas.
\begin{lemma}\label{lem.mult}
If a function $\mu: ]0, 1[^{k}\rightarrow\R$ is a multiplicative function
and
\[
\mu\left(x\right)=\mu\left(\mathbf{1}-x\right)
\]
holds for all $x\in ]0, 1[^{k}$, then $\mu\equiv 1$ or
$\mu\equiv 0$ on $]0, 1[^{n}$.
\end{lemma}
\begin{proof}
Due to Theorem \ref{ext.mult}. we obtain that the function $\mu$ is
uniquely extendable to $\R^{k}_{+}$. This extension is also denoted by $\mu$.
Thus
\[
\mu\left(x\right)=\mu\left(\mathbf{1}-x\right) \qquad \left(x\in \left. ]0,1[\right.^{k}\right)
\]
and
\[
\mu\left(xy\right)=\mu\left(x\right)\mu\left(y\right) \qquad \left(x, y\in\R^{k}_{+}\right)
\]

Let $p, q\in \R^{k}_{++}$ be arbitrary and substitute $x=\frac{p}{p+q}$ into the first equation, then using the fact that $\mu$ is multiplicative on $\R^{k}_{+}$, we get that
\[
\mu\left(\frac{p}{p+q}\right)=\mu\left(\frac{q}{p+q}\right),
\]
i.e.,
\[
\mu\left(p\right)\mu\left(\frac{\mathbf{1}}{p+q}\right)=
\mu\left(p\right)\mu\left(\frac{\mathbf{1}}{p+q}\right)
\]
holds for all $p, q\in \R^{k}_{+}$. This implies that either
\[
\mu\left(\frac{\mathbf{1}}{p+q}\right)=0,
\]
this means that $\mu\equiv 0$ (since in case $\mu$ vanishes at a point, then $\mu\equiv 0$), or
\[
\mu\left(p\right)=\mu\left(q\right)
\]
holds for all $p, q\in\R^{k}_{+}$, so $\mu$ has to be constant. Although using again the fact that $\mu$
is multiplicative, this constant has to be $1$.
\end{proof}

\begin{lemma}\label{lem.log}
If a function $l: ]0, 1[^{k}\rightarrow\R$ is a logarithmic function
and
\[
l\left(x\right)=l\left(\mathbf{1}-x\right)
\]
holds for all $x\in ]0, 1[^{k}$, then $l\equiv 0$ on $]0, 1[^{n}$.
\end{lemma}
\begin{proof}
Using Theorem \ref{ext.log}. we get immediately, that $l$ is uniquely extendable to $\R^{k}_{++}$.
We denote this extension also by $l$. Hence we have
\[
l\left(x\right)=l\left(\mathbf{1}-x\right) \qquad  \left(x\in \left.]0,1[\right.^{k}\right)
\]

and
\[
l\left(xy\right)=\left(x\right)+l\left(y\right). \qquad \left( x, y\in\R^{k}_{++}\right)
\]
Let $p, q\in\R^{k}_{++}$ be arbitrary, after substituting $x=\frac{p}{p+q}$
into the first equation,
then using that $l$ is logarithmic, we obtain that
\[
l\left(p\right)+l\left(\frac{\mathbf{1}}{p+q}\right)=l\left(q\right)+l\left(\frac{\mathbf{1}}{p+q}\right)
\]
holds for all $p, q\in\R^{k}_{++}$,
i.e.,
\[
l\left(p\right)=l\left(q\right)
\]
for all $p, q\in\R^{k}_{++}$. Hence the function $l$ is constant. However, using that
$l$ is logarithmic, we get that this constant has to be $0$, which had to be proved.
\end{proof}

\begin{theorem}\cite{Kuc85}
Let $\mu:\R^{k}_{+}\rightarrow\R$ be a multiplicative function, then
\[
\mu\left(x\right)=\mu\left(x_{1}, \ldots, x_{k}\right)=
\prod^{k}_{i=1}\mu_{i}\left(x_{i}\right)
\]
for all $x\in\R^{k}_{+}$, where each $\mu_{i}:\R_{+}\rightarrow\R$ is multiplicative, $i=1, \ldots, k$.
\end{theorem}

\begin{theorem}\cite{Kuc85}
Let $l:\R^{k}_{++}\rightarrow\R$ be a logarithmic function, then
\[
l\left(x\right)=l\left(x_{1}, \ldots, x_{k}\right)=\sum^{k}_{i=1}l_{i}\left(x_{i}\right)
\]
for all $x\in\R^{k}_{++}$, where each $l_{i}:\R_{++}\rightarrow\R$ is logarithmic.
\end{theorem}

\begin{theorem}\cite{Kuc85}
Let $\mu:\R_{+}\rightarrow\R$ be a continuous multiplicative function. Then, either there exist an
$\alpha\in\R$ such that
\[
\mu\left(x\right)=x^{\alpha}
\]
holds for all $x\in\R_{+}$ or $\mu\equiv 0$.
\end{theorem}

\begin{theorem}\cite{Kuc85}
Let $l:\R_{++}\rightarrow\R$ be a continuous logarithmic function. Then, either
\[
l\left(x\right)=\log\left(x\right) \quad
\left(x\in\R_{++}\right)
\]
with arbitrary basis for the logarithm, or $l$ is identically zero.
\end{theorem}

Now we present two additional results, namely the general solution of the fundamental equation of information of multiplicative type
(see \cite{ESS98} or \cite{AD75}) and the solution of a special associativity equation (see \cite{Mak00}).
\begin{theorem}\label{FEIM}\cite{ESS98}
Let $\mu:]0,1[^{k}\rightarrow\R$ be multiplicative. Then the general solution of the equation
\begin{equation}
h\left(x\right)+\mu\left(\mathbf{1}-x\right)h\left(\frac{y}{\mathbf{1}-x}\right)=
h\left(y\right)+\mu\left(\mathbf{1}-y\right)h\left(\frac{x}{\mathbf{1}-y}\right)
\end{equation}
on the set $\{(x, y)\in\R^{2k}_{++}\vert x, y, x+y\in ]0,1[^{k}\}$ is given by
\begin{equation}
h\left(x\right)=\mu\left(\mathbf{1}-x\right)l\left(\mathbf{1}-x\right)+
\mu\left(x\right)\left(l\left(x\right)+c\right) \quad (x\in ]0,1[^{k})
\end{equation}
in case $\mu$ is a projection; by
\begin{equation}
h\left(x\right)=l\left(\mathbf{1}-x\right)+c \qquad (x\in ]0,1[^{k})
\end{equation}
in case $\mu\equiv 1$; and finally by
\begin{equation}
h\left(x\right)=b\mu\left(\mathbf{1}-x\right)+c\mu\left(x\right)-b \quad (x\in ]0,1[^{k})
\end{equation}
in all other cases. Here $l:]0,1[^{k}\rightarrow\R$ is an arbitrary logarithmic function, and
$b, c\in\R$ are constants.
\end{theorem}
\noindent
The following lemma concerns a special associativity equation.
\begin{lemma}\label{ass.eq}\cite{Mak00}
Suppose that the function $A:\R^{2k}_{++}\rightarrow\R$ satisfies the functional equation
\begin{equation}
A\left(u+v, v\right)=
A\left(u, v+w\right) .
\qquad \left(u, v, w\in\R^{k}_{++}\right)
\end{equation}
Then there exists a function $\varphi:\R^{k}_{++}\rightarrow\R$ such that
\[
A\left(u, v\right)=\varphi\left(u+v\right)
\]
holds for all $u, v\in\R^{k}_{++}$, moreover if the function $A$ is continuous on $\R^{2k}_{++}$,
then the function $\varphi$ is continuous on $\R^{k}_{++}$.
\end{lemma}
The proof of this theorem runs similar as the proof of Lemma 1. in \cite{Mak00}, since the operations
as well as the ordering are to be done componentwise. All the same Lemma \ref{ass.eq}. concerns
a more special equation but on a more general domain, than the mentioned result of \cite{Mak00}.

\section{The modified entropy equation}

In this section we shall solve equation (\ref{gen.ent.eq}). Our main result is the following theorem.

\begin{theorem}
\label{main}
Let $\mu:\R^{k}_{++}\rightarrow\R$ be a multiplicative,
$f:\R^{3k}_{+}\rightarrow\R$ be a symmetric function and
assume that equation (\ref{gen.ent.eq}) holds for all $x, y, z\in\R^{k}_{++}$. Then,\\
\noindent
in case $\mu$ is a projection, there exist functions
	$l, \psi_{1}:\R^{k}_{++}\rightarrow\R$ such that $l$ is a
	logarithmic function on $\R^{k}_{++}$ and
	\begin{equation}
	f\left(x, y, z\right)=\mu\left(x\right)l\left(x\right)+\mu\left(y\right)l\left(y\right)+\mu\left(z\right)l\left(z\right)+
	\psi_{1}\left(x+y+z\right)
	\end{equation}
	holds for all $x, y, z\in\R^{k}_{+}$, \\
in case $\mu\equiv 1$, there exists a function $\psi_{2}:\R^{k}_{++}\rightarrow\R$
	such that
	\begin{equation}
	f\left(x, y, z\right)=\psi_{2}\left(x+y+z\right)
	\end{equation}
	holds for all $x, y, z\in\R^{k}_{++}$, \\
 in all other cases, there exist a function $\psi_{3}:\R^{k}_{++}\rightarrow\R$ and a constant
	$b\in\R$ such that
	\begin{equation}
	f\left(x, y, z\right)=
	 b\left(\mu\left(x\right)+\mu\left(y\right)+\mu\left(z\right)\right)+\psi_{3}\left(x+y+z\right)
	\end{equation}
	holds for all $x, y, z\in\R^{k}_{++}$.

\end{theorem}

\begin{proof}
Assume that (\ref{gen.ent.eq}) holds for all $x, y, z\in\R^{k}_{++}$.
Define the functions $F$ and $h$ as follows
\begin{equation}\label{Fdef}
F\left(u, v\right)=f\left(\mathbf{0}, u, v\right)  \quad (u, v\in\R^{k}_{++})
\end{equation}
and
\begin{equation}\label{hdef}
h\left(t\right)=F\left(\mathbf{1}-t, t\right). \quad (t\in ]0, 1[^{k})
\end{equation}
Then due to the symmetry of the function $f$ we get that
\begin{equation}\label{Fsym}
F\left(u, v\right)=F\left(v, u\right)\quad (u, v\in\R^{k}_{++})
\end{equation}
and
\begin{equation}\label{hsym}
h\left(t\right)=h\left(\mathbf{1}-t\right). \quad (t\in ]0, 1[^{k})
\end{equation}
Furthermore it follows from (\ref{gen.ent.eq}), (\ref{Fdef}) and (\ref{hdef}) that
\begin{equation}\label{eq8}
f\left(x, y, z\right)=F\left(x, y+z\right)+\mu\left(y+z\right)h\left(\frac{y}{y+z}\right)
\end{equation}
holds for all $x, y, z\in\R^{k}_{++}$.
On the other hand the symmetry of $f$ implies that
\begin{equation}\label{eq9}
F\left(x, y+z\right)+\mu\left(y+z\right)h\left(\frac{y}{y+z}\right)=
F\left(y, x+z\right)+\mu\left(x+z\right)h\left(\frac{x}{x+z}\right)
\end{equation}
for all $x, y, z\in\R^{k}_{++}$. \\
\noindent
Define the set $D=\left\{(x, y)\in\R^{2k}_{++}\vert x, y, x+y\in ]0, 1[^{k}\right\}$. Let
$(x, y)\in D$ and substitute $z=\mathbf{1}-x-y$ into (\ref{eq9}). Then by (\ref{hdef}) we obtain that
\begin{equation}
h\left(x\right)+\mu\left(\mathbf{1}-x\right)h\left(\frac{y}{\mathbf{1}-x}\right)=
h\left(y\right)+\mu\left(\mathbf{1}-y\right)h\left(\frac{x}{\mathbf{1}-y}\right)
\end{equation}
holds for all $(x, y)\in D$, i.e., $h$ satisfies the fundamental equation of information of
multiplicative type on $D$.
This means by Theorem \ref{FEIM}. and by the results of Section 4.3. in \cite{ESS98} that
\[
h\left(x\right)=\mu\left(\mathbf{1}-x\right)l\left(\mathbf{1}-x\right)+
\mu\left(x\right)\left(l\left(x\right)+c\right) \qquad (x\in ]0,1[^{k})
\]
in case $\mu$ is a projection;
\[
h\left(x\right)=l\left(\mathbf{1}-x\right)+c \qquad (x\in ]0,1[^{k})
\]
in case $\mu\equiv 1$; finally
\[
h\left(x\right)=b\mu\left(\mathbf{1}-x\right)+c\mu\left(x\right)-b \qquad (x\in ]0,1[^{k})
\]
in all other cases. Here $l:]0,1[^{k}\rightarrow\R$ is an arbitrary logarithmic function and
$b, c\in\R$ are constants. \\
So, concerning the function $\mu$ we have to distinguish three cases. \\
\noindent
\textsc{Case I.} ($\mu$ is a projection)\\
\noindent
First note that due to (\ref{hsym}) we have $c=0$.
Indeed, $h\left(x\right)=h\left(\mathbf{1}-x\right)$ on $]0, 1[^{k}$ implies that
	\[
	\begin{array}{l}
	\mu\left(\mathbf{1}-x\right)l\left(\mathbf{1}-x\right)+
	\mu\left(x\right)l\left(x\right)+c\mu\left(x\right) \\
	=\mu\left(x\right)l\left(x\right)+
	\mu\left(\mathbf{1}-x\right)l\left(\mathbf{1}- x\right)+c\mu\left(\mathbf{1}-x\right)
	\end{array}
\]
holds for all $x\in]0, 1[^{k}$, but this means that
\[
c\mu\left(x\right)=c\mu\left(\mathbf{1}-x\right)
\]
for all $x\in ]0, 1[^{k}$, i.e., either $c=0$ or $\mu\left(x\right)=\mu\left(\mathbf{1}-x\right)$ on
$]0, 1[^{k}$, but due to Lemma \ref{lem.mult}. this implies $\mu\equiv 1$ or
$\mu\equiv 0$. In case $\mu$ is a non identically zero projection, we get that
\[
h\left(x\right)=\mu\left(x\right)l\left(x\right)+\mu\left(\mathbf{1}-x\right)l\left(\mathbf{1}-x\right)
\]
holds for all $]0, 1[^{k}$. Therefore
\[
\begin{array}{l}
\mu\left(y+z\right)h\left(\frac{y}{y+z}\right)
 = \mu\left(y+z\right)\left(\mu\left(\frac{z}{y+z}\right)l\left(\frac{z}{y+z}\right)+
\mu\left(\frac{y}{y+z}\right)l\left(\frac{y}{y+z}\right)\right)\\
 = \mu\left(y+z\right)\mu\left(\frac{z}{y+z}\right)l\left(\frac{z}{y+z}\right)+
\mu\left(y+z\right)\mu\left(\frac{y}{y+z}\right)l\left(\frac{y}{y+z}\right)\\
 = \mu\left(z\right)l\left(\frac{z}{y+z}\right)+\mu\left(y\right)l\left(\frac{y}{y+z}\right).
\end{array}
\]
Thus by (\ref{eq9})
\[
\begin{array}{l}
F\left(x, y+z\right)+\mu\left(z\right)l\left(\frac{z}{y+z}\right)+\mu\left(y\right)l\left(\frac{y}{y+z}\right)\\
=
F\left(y, x+z\right)+\mu\left(z\right)l\left(\frac{z}{x+z}\right)+\mu\left(x\right)l\left(\frac{x}{x+z}\right).
\end{array}
\]
Using that $l$ is logarithmic and $\mu$ is a projection, i.e., additive and multiplicative, we obtain that
\[
\begin{array}{l}
F\left(x, y+z\right)+\mu\left(z\right)l\left(z\right)\\
-\mu\left(z\right)l\left(y+z\right)+
\mu\left(y\right)l\left(y\right)-\mu\left(y\right)l\left(y+z\right)\\
=F\left(y, x+z\right)+\mu\left(z\right)l\left(z\right)\\
-\mu\left(z\right)l\left(x+z\right)+
\mu\left(x\right)l\left(x\right)-\mu\left(x\right)l\left(x+z\right)
\end{array}
\]
holds for all $x, y, z\in\R^{k}_{++}$, i.e.,
\begin{equation}
\begin{array}{l}
F\left(x, y+z\right)-\mu\left(x\right)l\left(x\right)-\mu\left(y+z\right)l\left(y+z\right)\\
=
F\left(y, x+z\right)-\mu\left(y\right)l\left(y\right)-\mu\left(x+z\right)l\left(x+z\right).
\end{array}
\end{equation}
Let
\begin{equation}\label{Hdef}
H\left(x, y\right)=F\left(x, y\right)-\mu\left(x\right)l\left(x\right)-\mu\left(y\right)l\left(y\right),
\quad(x, y\in\R^{k}_{++})
\end{equation}
then
\[
H\left(x, y+z\right)=H\left(y, x+z\right)
\]
holds for all $x, y, z\in\R^{k}_{++}$, which is a special associativity equation. Thus by Lemma \ref{ass.eq}.
we get that
there exists a function $\psi_{1}:\R^{k}_{++}\rightarrow\R$ such that
\[
H\left(x, y\right)=\psi_{1}\left(x+y\right)
\]
for all $x, y\in\R^{k}_{++}$.
Because of (\ref{eq8}) and (\ref{Hdef}) this means that
\[
f\left(x, y, z\right)=\mu\left(x\right)l\left(x\right)+
\mu\left(y\right)l\left(y\right)+
\mu\left(z\right)l\left(z\right)+
\psi_{1}\left(x+y+z\right)
\]
is fulfilled for all $x, y, z\in\R^{k}_{++}$. \\
\noindent
In case $\mu\equiv 0$ we get immediately from (\ref{eq9}) a special associativity equation, namely
\[
F\left(x, y+z\right)=
F\left(y, x+z\right).
\quad \left(x, y, z\in\R^{k}_{++}\right)
\]
Thus by Lemma \ref{ass.eq}. there exists a function
$\Psi_{1}:\R^{k}_{++}\rightarrow\R$ such that
\[
F\left(x, y\right)=\Psi_{1}\left(x+y\right).
\quad \left(x, y\in\R^{k}_{++}\right)
\]
Since $\mu\equiv 0$, this implies that
\[
f\left(x, y, z\right)=
\mu\left(x\right)l(x)+\mu\left(y\right)l(y)+\mu\left(z\right)l(z)+\Psi_{1}\left(x+y+z\right)
\]
holds for all $x, y, z\in\R^{k}_{++}$, with some logarithmic function
$l:\R^{k}_{++}\rightarrow\R$.\\
\noindent
\textsc{Case II. } ($\mu\equiv 1$)\\
\noindent
In this case
\[
h\left(x\right)=l\left(\mathbf{1}-x\right)+c
\]
holds on $]0,1[^{k}$. However, by equation (\ref{hsym}) we get that
\[
l\left(\mathbf{1}-x\right)+c=l\left(x\right)+c \quad
\]
Since $l$ is logarithmic by Lemma \ref{lem.log}. this means that
$l$ is identically zero on $]0,1[^{k}$. Thus
\[
h\left(x\right)=c, \qquad \left(x\in ]0,1^{k}[\right)
\]
with a constant $c\in \R$.
Now using equation (\ref{eq9}) we obtain that
\begin{equation}\label{eq.ass.1}
F\left(x, y+z\right)=F\left(y, x+z\right)
\end{equation}
holds for all $x, y, z\in\R^{k}_{++}$. Again we get a special associativity equation. Hence
by Lemma \ref{ass.eq}. there exists a function $\psi^{*}_{2}:\R^{k}_{++}\rightarrow\R$ such that
\[
F\left(x, y\right)=\psi^{*}_{2}\left(x+y\right).
\]
Together with equation (\ref{eq8}) this means that
\[
f\left(x, y, z\right)=\psi^{*}_{2}\left(x+y+z\right)+c
\]
holds for all $x, y, z\in\R^{k}_{++}$.
Define the function $\psi_{2}$ on $\R^{k}_{++}$ by
\[
\psi_{2}(x)=\psi^{*}_{2}(x)+c,  \qquad \left(x\in \R^{k}_{++}\right)
\]
then we obtain that
\[
f\left(x, y, z\right)=\psi_{2}\left(x+y+z\right)
\]
holds for all $x, y, z \in\R^{k}_{++}$.
\\
\noindent
\textsc{Case III. } ($\mu$ is neither a projection nor identically $1$)\\
\noindent
Since the function $h$ is symmetric to the point
$\mathbf{\frac{1}{2}}\in\R^{k}_{++}$, we get that
\[
b\mu\left(\mathbf{1}-x\right)+c\mu\left(x\right)-b=
b\mu\left(x\right)+c\mu\left(\mathbf{1}-x\right)-b
\]
holds for all $x \in ]0,1[^{k}$, where $b, c\in\R$ are constants.
This means that
\[
\left(b-c\right)\mu\left(\mathbf{1}-x\right)=\left(b-c\right)\mu\left(x\right)
\quad (x\in ]0,1[^{k})
\]
Thus either $b=c$ or $\mu\left(\mathbf{1}-x\right)=\mu\left(x\right)$ holds
for all $x\in ]0,1[^{k}$. Although the last statement implies that
$\mu$ is identically $1$ or
identically $0$ on $]0,1[^{k}$ (Lemma \ref{lem.mult}.), which contradicts to the assumption that
$\mu$ in neither identically $1$ nor a projection. So $b=c$. Thus we obtain that
\begin{equation}\label{eq11}
\mu\left(y+z\right)h\left(\frac{y}{y+z}\right)=b\mu\left(z\right)+b\mu\left(y\right).
\qquad (y, z\in\R^{k}_{++})
\end{equation}
Therefore by (\ref{eq9})
\begin{equation}
F\left(x, y+z\right)+b\mu\left(z\right)+b\mu\left(y\right)=
F\left(y, x+z\right)+b\mu\left(z\right)+b\mu\left(x\right)
\end{equation}
for all $x, y, z\in\R^{k}_{++}$ that is,
\begin{equation}
F\left(x, y+z\right)-b\mu\left(x\right)=
F\left(y, x+z\right)-b\mu\left(y\right)
\end{equation}
holds for all $x, y, z\in\R^{k}_{++}$.
Define the function $H$ on $\R^{2k}_{++}$ as follows
\begin{equation}\label{Hdef2}
H\left(x, y\right)=F\left(x, y\right)-b\mu\left(x\right).
\end{equation}
Then $H$ satisfies on $\R^{2k}_{++}$ a special associativity equation, namely
\begin{equation}
H\left(x, y+z\right)=H\left(y, x+z\right). \qquad (x, y, z\in\R^{k}_{++})
\end{equation}
So by Lemma \ref{ass.eq}. there exists a function $\psi_{3}:\R^{k}_{++}\rightarrow\R$ such that
\[
H\left(x, y\right)=\psi_{3}\left(x+y\right)
\]
holds for all $x, y\in\R^{k}_{++}$, together with  equations (\ref{eq8}) and (\ref{eq11}) this means that
\begin{equation}
f\left(x, y, z\right)=b\left(\mu\left(x\right)+\mu\left(y\right)+\mu\left(z\right)\right)+
\psi_{3}\left(x+y+z\right)
\end{equation}
holds for all $x, y, z\in\R^{k}_{++}$.
\\
\noindent
This ends the proof since concerning the function $\mu$ all the cases were covered.
\end{proof}

\section{Regular solutions of equation (\ref{gen.ent.eq}) and other corollaries}

The last section of the present paper concerns the regular solutions of
the modified entropy equation and we will show that equation
(\ref{gen.ent.eq}) is a generalization of equation
posed in \cite{AGMS08} as well as the entropy equation.

\begin{theorem}\label{Thm2.1}
Let $\mu:\R^{k}_{++}\rightarrow\R$ be a multiplicative function.
If the function $f:\R^{k}_{++}\rightarrow\R$ is symmetric, continuous on $\R^{3}_{++}$ and
satisfies (\ref{gen.ent.eq}) for all $x, y, z\in\R^{k}_{++}$, then in case $\mu$ is a projection,
there exists a continuous function $\Psi_{1}:\R^{k}_{++}\rightarrow\R$ such that
\begin{equation}\label{E3.1.1.}
f\left(x, y, z\right)=\mu\left(x\right)\log\left(x\right)+
\mu\left(y\right)\log\left(y\right)+
\mu\left(z\right)\log\left(z\right)+
\Psi_{1}\left(x+y+z\right)
\end{equation}
holds for all $x, y, z\in\R^{k}_{++}$, with arbitrary basis for the logarithm, \\
\noindent
in case $\mu\equiv 1$, there exist a continuous function
$\Psi_{2}:\R^{k}_{++}\rightarrow\R$ such that
\begin{equation} \label{E3.1.2.}
f\left(x, y, z\right)=\Psi_{2}\left(x+y+z\right)
\end{equation}
holds for all $x, y, z\in\R^{k}_{++}$, \\
\noindent
in all other cases, there exist a continuous function $\Psi_{3}:\R^{k}_{++}\rightarrow\R$
and a constant $b\in\R$ such that
\begin{equation}\label{E3.1.3.}
f\left(x, y, z\right)=b\left(\mu\left(x\right)+
\mu\left(y\right)+
\mu\left(z\right)\right)+
\Psi_{3}\left(x+y+z\right).
\end{equation}
holds for all $x, y, z\in\R^{k}_{++}$.
\end{theorem}

\begin{proof}
First let us observe that in case the function
$f:\R^{3k}\rightarrow\R$ is continuous and satisfies
equation (\ref{gen.ent.eq}), then either the multiplicative function
$\mu:\R^{k}_{++}\rightarrow\R$ is continuous or the map
$x\mapsto f\left(\mathbf{0}, \mathbf{1}-x, x\right)$ is identically zero on $]0,1[^{k}$, due
to the representation (\ref{eq8}).
Therefore, in the first case, if we define the functions $F$ and $h$ by (\ref{Fdef}) and (\ref{hdef}) on
$\R^{2k}_{++}$ and on $]0,1[^{k}$, respectively,
then they will be continuous on their domain, as well. \\
\noindent
Hence the function $h:]0,1[^{k}\rightarrow\R$ is continuous and satisfies the
fundamental equation of information on $D$, by Theorem \ref{FEIM}. and
the by the results of Section 4.3. in \cite{ESS98}
\[
h\left(x\right)=\mu\left(\mathbf{1}-x\right)\log\left(\mathbf{1}-x\right)+
\mu\left(x\right)\left(\log\left(x\right)+c\right),
\quad \left(x\in ]0,1[^{k}\right)
\]
in case $\mu$ is a projection;
\[
h\left(x\right)=\log\left(\mathbf{1}-x\right)+c
\quad \left(x\in ]0,1[^{k}\right)
\]
in case $\mu\equiv 1$;
\[
h\left(x\right)=b\mu\left(\mathbf{1}-x\right)+c\mu\left(x\right)-b
\quad \left(x\in ]0,1[^{k}\right)
\]
in all other cases, where $\log:\R^{k}_{++}\rightarrow\R$ denotes
an arbitrary continuous logarithmic function and
$b, c\in\R$ are constants. \\
\noindent
A same argument as in the proof of Theorem \ref{main}. can be made to get
\[
h\left(x\right)=\mu\left(\mathbf{1}-x\right)\log\left(\mathbf{1}-x\right)+
\mu\left(x\right)\log\left(x\right),
\quad \left(x\in ]0,1[^{k}\right)
\]
in case $\mu$ is a projection;
\[
h\left(x\right)=c
\quad \left(x\in ]0,1[^{k}\right)
\]
in case $\mu\equiv 1$;
\[
h\left(x\right)=b\mu\left(\mathbf{1}-x\right)+b\mu\left(x\right)-b
\quad \left(x\in ]0,1[^{k}\right)
\]
in all other cases, due to the property
\[
h\left(x\right)=h\left(\mathbf{1}-x\right).
\quad \left(x\in ]0,1[^{k}\right)
\]
On the other hand the function $F$ is continuous and concerning the function
$\mu$ we can define $H$ by (\ref{Hdef}), (\ref{eq.ass.1}) and (\ref{Hdef2})
to get special associativity equations.
This implies by Lemma \ref{ass.eq}. that the functions,
$\Psi_{1}, \Psi_{2}, \Psi_{3}$ occuring in Theorem \ref{main}. are continuous. \\
\noindent
In case the map $x\mapsto f\left(\mathbf{0}, \mathbf{1}-x, x\right)$ is identically zero on
$]0,1[^{k}$, then we get from (\ref{gen.ent.eq}), that
\begin{equation}\label{E3.1.4.}
f\left(x, y, z\right)=f\left(x, y+z, \mathbf{0}\right).
\quad \left(x, y, z\in\R^{k}_{++}\right)
\end{equation}
Define the function $F$ on $\R^{2k}_{++}$ by (\ref{Fdef}),
after interchanging $x$ and $y$ in the previous equation and
using the symmetry of $f$, we obtain that
\[
F\left(x, y+z\right)=F\left(y, x+z\right)=
F\left(z, x+y\right)=F\left(x+y, z\right)
\]
holds for all $x, y, z\in\R^{k}_{++}$. Thus by Lemma \ref{ass.eq} there exists a function
$\Psi:\R^{k}_{++}\rightarrow\R$ such that
\[
F\left(x, y\right)=\Psi\left(x+y\right)
\]
for all $x, y, z\in\R^{k}_{++}$. Finally, using (\ref{E3.1.4.}) we get that
\[
f\left(x, y, z\right)=\Psi\left(x+y+z\right) .
\quad \left(x, y, z\in\R^{k}_{++}\right)
\]

Hence all in all, we obtain that equations (\ref{E3.1.1.}), (\ref{E3.1.2.}) and (\ref{E3.1.3.}) are fulfilled.
\end{proof}

\begin{remark}
To get the same result as in the previous theorem we do not need to assume that the function $f$ is
continuous,  but it is enough to suppose that the map
\[
\left(x, y\right)\longmapsto f\left(\mathbf{0}, x, y\right)
\]
is continuous on $\R^{2k}_{++}$.
\end{remark}

\begin{proof}
Assume that the map
\[
\left(x, y\right)\longmapsto f\left(\mathbf{0}, x, y\right)
\]
is continuous on $\R^{2k}_{++}$, and the function $f:\R^{3k}_{++}\rightarrow\R$
satisfies equation (\ref{gen.ent.eq}).
Then the functions $F$ and $h$ defined by (\ref{Fdef}) and (\ref{hdef}), respectively,
are continuous. Thus Theorem \ref{Thm2.1} can be applied to get the stated result.
\end{proof}

In what follows we present and prove three additional results which are corollaries of our main result.
This theorems can be found in \cite{AGMS08}, \cite{Acz77} and in \cite{KM74}, respectively.

\begin{theorem}\cite{AGMS08}
Let $f:\R_{++}\rightarrow\R$ be a symmetric function and suppose that
\[
f\left(x, y, z\right)=f\left(x, y+z, 0\right)+
\left(y+z\right)f\left(0, \frac{y}{y+z}, \frac{z}{y+z}\right)
\]
holds for all $x, y, z\in\R_{++}$.
Then there exist function $\varphi, \psi:\R_{++}\rightarrow\R$ such that
\[
\varphi\left(x y\right)=x\varphi\left(y\right)+y\varphi\left(x\right)
\qquad \left(x, y\in\R_{++}\right)
\]
and
\[
f\left(x, y, z\right)=\varphi\left(x\right)+
\varphi\left(y\right)+
\varphi\left(z\right)+
\psi\left(x+y+z\right)
\]
for all $x, y, z\in\R_{++}$.
\end{theorem}

\begin{proof}
Let us observe that the equation above is a special case of equation
(\ref{gen.ent.eq}) with $k=1$ and with the multiplicative function
\[
\mu\left(x\right)=x,
\qquad
\left(x\in\R_{++}\right)
\]
which is obviously a projection from $\R_{++}$ to $\R_{++}$.
Thus by Theorem \ref{main}. we get that there exist functions
$l, \psi:\R_{++}\rightarrow\R$ such that $l$ is
logarithmic on $\R_{++}$ and
\[
f\left(x, y, z\right)=xl\left(x\right)+
yl\left(y\right)+zl\left(z\right)+\psi\left(x+y+z\right)
\]
holds for all $x, y, z\in\R_{++}$. Define the function
$\varphi$ on $\R_{++}$ by
\[
\varphi\left(x\right)=xl\left(x\right),  \qquad
\left(x\in\R_{++}\right)
\]
then we have
\[
\varphi\left(x y\right)=x\varphi\left(y\right)+y\varphi\left(x\right)
\qquad \left(x, y\in\R_{++}\right)
\]
and
\[
f\left(x, y, z\right)=\varphi\left(x\right)+
\varphi\left(y\right)+
\varphi\left(z\right)+
\psi\left(x+y+z\right)
\]
for all $x, y, z\in\R_{++}$, which had to be proved.
\end{proof}

\begin{theorem}\label{aczel}\cite{Acz77}
Let
\[
\mathcal{D}=\left\{(x, y, z)\in\R^{3}\vert x\geq 0, y\geq 0, z\geq 0, x+y+z>0\right\}.
\]
Assume that the function $f:\mathcal{D}\rightarrow\R$ is symmetric, positively homogeneous of
degree $1$, satisfies the functional equation
\[
f\left(x, y, z\right)=f\left(x+y, z, 0\right)+f\left(x, y, 0\right)
\]
on
\[
S=\left\{\left(x, y, z\right)\in\R^{3}\vert x\geq 0, y\geq 0, z\geq 0, xy+yz+zx>0\right\}
\]
and the map $x\longmapsto f\left(1-x, x, 0\right)$ is either continuous at a point
or bounded on an interval or integrable on
the closed subintervals of $]0,1[$ or measurable on $]0,1[$. Then, and only then
\[
f\left(x, y, z\right)=
x\log\left(x\right)+y\log\left(y\right)+z\log\left(z\right)-
\left(x+y+z\right)\log\left(x+y+z\right)
\]
holds for all $\left(x, y, z\right)\in \mathcal{D}$, with arbitrary
basis for the logarithm and with the convention $0\cdot \log\left(0\right)=0$.
\end{theorem}
\begin{proof}
The restriction $xy+yz+zx>0$ in the definition of the set
$S$ excludes the case where two variables are $0$. If we define
\[
f\left(x, 0, 0\right)=f\left(0, x, 0\right)=f\left(0, 0, x\right),  \quad \left(x\in\R_{++}\right)
\]
then $f$ remains symmetric and all the assumptions of the theorem remain true for all
$\left(x, y, z\right)\in\mathcal{D}$.
Because of the symmetry of the function $f$ we obtain that
\[
f\left(x, y, z\right)=f\left(x, y+z, 0\right)+f\left(0, y, z\right)
\]
holds for all $x, y, z\in\R_{++}$, too.
Since $f$ is homogeneous of degree $1$ we obtain the following equation
\[
f\left(x, y, z\right)=f\left(x, y+z, 0\right)+
\left(y+z\right)f\left(0, \frac{y}{y+z}, \frac{z}{y+z}\right),
\]
where $x, y, z\in\R_{+}$ are such that $y+z>0$.
Applying Theorem \ref{main}. and the results of Section 3.2. in \cite{AD75}
we get that there exists a function $\psi_{1}:\R_{++}\rightarrow\R$
such that
\begin{equation}\label{acz.eq}
f\left(x, y, z\right)=
x\log\left(x\right)+y\log\left(y\right)+z\log\left(z\right)+
	\psi_{1}\left(x+y+z\right)
\end{equation}
holds for all $x, y, z\in\R_{++}$.
At this point of our proof we will show that
\[
\psi_{1}\left(x\right)=-x\log\left(x\right)
\]
holds for all $x\in\R_{++}$.
Using that $f$ is homogeneous of degree $1$ we obtain that
\begin{equation}\label{eq12}
\begin{array}{l}
\lambda x\log\left(x\right)+
\lambda y\log\left(y\right)+
\lambda z\log\left(z\right)+
\lambda \psi_{1}\left(x+y+z\right)=\\
\lambda x\log\left(x\right)+
\lambda y\log\left(y\right)+
\lambda z\log\left(z\right)+\\
\lambda \left(x+y+z\right)\log\left(\lambda\right)+
\psi_{1}\left(\lambda\left(x+y+z\right)\right)
\end{array}
\end{equation}
holds for all $\lambda, x, y, z\in\R_{++}$. Substitute into (\ref{eq12}) $x=y=z=\frac{1}{3}$,
then we get that
\[
\psi_{1}\left(\lambda\right)=-\lambda \log\left(\lambda\right) +\lambda \psi_{1}\left(1\right) .
\quad
\left(\lambda\in\R_{++}\right)
\]
This means that
\[
\begin{array}{l}
f\left(x, y, z\right)=x\log\left(x\right)+
y\log\left(y\right)+
z\log\left(z\right)-\\
\left(x+y+z\right)\log\left(x+y+z\right)-\left(x+y+z\right)\psi_{1}\left(1\right)
\end{array}
\]
for all $x, y, z\in\R_{++}$.
On the other hand $f$ satisfies the entropy equation, therefore $\psi_{1}\left(1\right)=0$.
This implies that
\[
f\left(x, y, z\right)=x\log\left(x\right)+
y\log\left(y\right)+
z\log\left(z\right)-
\left(x+y+z\right)\log\left(x+y+z\right)
\]
holds for all $\left(x, y, z\right)\in\mathcal{D}$.
\end{proof}

\begin{corollary}\cite{KM74}
If a function $f$ is continuous, symmetric and positively homogeneous of degree $1$ on the set
\[
S=\left\{\left(x, y, z\right)\in\R^{3}\vert x\geq 0, y\geq 0, z\geq 0, xy+yz+zx>0\right\}
\]
and satisfies the functional equation
\[
f\left(x, y, z\right)=f\left(x+y, 0, z\right)+f\left(x, y, 0\right)
\]
in the interior $S^{\circ}$ of $S$, then
\[
f\left(x, y, z\right)=
x\log\left(x\right)+y\log\left(y\right)+z\log\left(z\right)-
\left(x+y+z\right)\log\left(x+y+z\right),
\]
where the basis of the logarithm is arbitrary.
\end{corollary}
\begin{proof}
As it can be read in \cite{Acz77}, this statement is a consequence of the previous theorem.
Thus Theorem \ref{aczel}. can be applied directly to get the continuous solutions of the equation in question.
\end{proof}


{\bf Acknowledgments:}
This result has been supported by the Hungarian Scientific Research Fund (OTKA)
Grant NK 68048. The author is beholden to Gyula Maksa for his helpful comments.

\end{document}